\documentclass[12pt]{article}

\usepackage{amsmath,amssymb,amsthm,xspace,amscd}

\theoremstyle{remark}

\newtheorem{remark}{Remark}

\theoremstyle{plain}

\newtheorem{theorem}{Theorem}

\newtheorem{proposition}{Proposition}

\newtheorem{lemma}{Lemma}

\newtheorem*{theorem_empty}{Theorem}

\newcommand{\eps}{\varepsilon}

\newcommand{\unitcircle}{{{\mathbb T}^1}}
\newcommand{\R}{{\mathbb R}}
\newcommand{\Z}{{\mathbb Z}}
\newcommand{\DR}{{\rm D}}
\newcommand{\Cr}{{\rm Cr}}
\newcommand{\Dist}{{\rm Dist}}
\newcommand{\Obig}{{\mathcal O}}
\newcommand{\bp}{{\bar p}}
\newcommand{\Pcal}{\mathcal P}

\title{Herman's Theory Revisited (Extension)}
\author{
    A.~Teplinsky\thanks{Institute of Mathematics, Kiev, Ukraine}
    }
\begin{document}

\maketitle
\begin{abstract}
We prove that a $C^{3+\beta}$-smooth orientation-preserving circle diffeomorphism with rotation number in
Diophantine class $D_\delta$, $0<\beta<\delta<1$, is $C^{2+\beta-\delta}$-smoothly conjugate to a rigid rotation.
\end{abstract}

\section{Introduction}

An irrational number $\rho$ is said to belong to Diophantine class
$D_\delta$ if there exists a constant $C>0$ such that
$|\rho-p/q|\ge Cq^{-2-\delta}$ for any rational number $p/q$.

In~\cite{KT-jams}, the following result was proven.
\begin{theorem_empty}[Khanin-T.] Let $T$ be a $C^{2+\alpha}$-smooth orientation-preserving circle
diffeomorphism with rotation number $\rho\in D_\delta$, $0<\delta<\alpha\le1$. Then $T$ is
$C^{1+\alpha-\delta}$-smoothly conjugate to the rigid rotation by angle $\rho$.
\end{theorem_empty}

By the smoothness of conjugacy we mean the smoothness of the homeomorphism $\phi$ such that
\begin{equation}
\phi\circ T\circ\phi^{-1}=R_\rho,
\end{equation}
where $R_\rho(\xi)=\xi+\rho\mod1$ is the mentioned rigid rotation.

The aim of the present paper is to extend the Theorem above to the case of $T\in C^{3+\beta}$,
$0<\beta<\delta<1$, so that the extended result is read as follows:
\begin{theorem} Let $T$ be a $C^r$-smooth orientation-preserving circle
diffeomorphism with rotation number $\rho\in D_\delta$,
$0<\delta<1$, $2+\delta<r<3+\delta$. Then $T$ is
$C^{r-1-\delta}$-smoothly conjugate to the rigid rotation by angle
$\rho$.
\end{theorem}

Historically, the first global results on smoothness of conjugation with rotations were obtained by
M.~Herman~\cite{H}. Later J.-C.~Yoccoz extended the theory to the case of Diophantine rotation numbers~\cite{Y}.
The result, recognized generally as the final answer in the theory, was proven by Y.~Katznelson,
D.~Ornstein~\cite{KO1}. In our terms it states that the conjugacy is $C^{r-1-\delta-\eps}$-smooth for any
$\eps>0$ provided that $0<\delta<r-2$. Notice that Theorem~1 is stronger than the result just cited, though valid
for a special scope of parameter values only, and it is sharp, i.e. smoothness of conjugacy higher than
$C^{r-1-\delta}$ cannot be achieved in general settings, as it follows from the examples constructed in
\cite{KO1}. At present, we do not know whether Theorem~1 can be extended further, and the examples mentioned do
not prevent such an extension.

In paper by K.~Khanin, Ya.~Sinai~\cite{SK}, published simultaneously with~\cite{KO1}, similar problems were
approached by a different method. The method we use is different from the one of~\cite{KO1}; it is based on the
ideas of~\cite{SK}, the cross-ratio distortion tools and certain exact relations between elements of the
dynamically generated structure on the circle.

All the implicit constants in asymptotics written as $\Obig(\cdot)$ depend on the function $f$ only in Section~2
and on the diffeomorphism $T$ only in Section~3.

\section{Cross-ratio tools}

The {\em cross-ratio} of four pairwise distinct points $x_1, x_2, x_3, x_4$ is
$$
\Cr(x_1,x_2,x_3,x_4)=\frac{(x_1-x_2)(x_3-x_4)}{(x_2-x_3)(x_4-x_1)}
$$
Their {\em cross-ratio distortion} with respect to a strictly increasing function $f$ is
$$
\Dist(x_1,x_2,x_3,x_4;f)=\frac{\Cr(f(x_1),f(x_2),f(x_3),f(x_4))}{\Cr(x_1,x_2,x_3,x_4)}
$$
Clearly,
\begin{equation}\label{eq:Dist=DRtimesDR}
\Dist(x_1,x_2,x_3,x_4;f)=\frac{\DR(x_1,x_2,x_3;f)}{\DR(x_1,x_4,x_3;f)},
\end{equation}
where
$$
\DR(x_1,x_2,x_3;f)=\frac{f(x_1)-f(x_2)}{x_1-x_2}:\frac{f(x_2)-f(x_3)}{x_2-x_3}
$$
is the {\em ratio distortion} of three distinct points $x_1,x_2,x_3$ with respect to $f$.

In the case of smooth $f$ such that $f'$ does not vanish, both the ratio distortion and the cross-ratio distortion
are defined for points, which are not necessarily pairwise distinct, as the appropriate limits (or, just by
formally replacing ratios $(f(a)-f(a))/(a-a)$ with $f'(a)$ in the definitions above).

Notice that both ratio and cross-ratio distortions are multiplicative with respect to composition: for two
functions $f$ and $g$ we have
\begin{equation}\label{eq:mult-DR}
\DR(x_1,x_2,x_3;f\circ g)=\DR(x_1,x_2,x_3;g)\cdot\DR(g(x_1),g(x_2),g(x_3);f)
\end{equation}
\begin{equation}\label{eq:mult-Dist}
\Dist(x_1,x_2,x_3,x_4;f\circ g)=\Dist(x_1,x_2,x_3,x_4;g)\cdot\Dist(g(x_1),g(x_2),g(x_3),g(x_4);f)
\end{equation}

For $f\in C^{3+\beta}$ it is possible to evaluate the next entry
in the asymptotical expansions for both ratio and cross-ratio
distortions. The {\em Swartz derivative} of $C^{3+\beta}$-smooth
function is defined as
$Sf=\frac{f'''}{f'}-\frac{3}{2}(\frac{f''}{f'})$.

\begin{proposition}\label{prop:DR}
Let $f\in C^{3+\beta}$, $\beta\in[0,1]$, and $f'>0$ on $[A,B]$.
Then for any $x_1, x_2, x_3\in[A,B]$ the following estimate holds:
\begin{equation}\label{eq:propDR}
\DR(x_1,x_2,x_3;f)=1+(x_1-x_3)\left(\frac{f''(x_1)}{2f'(x_1)}+\frac{1}{6}Sf(x_1)(x_2+x_3-2x_1)
+\Obig(\Delta^{1+\beta})\right),
\end{equation}
where $\Delta=\max\{x_1,x_2,x_3\}-\min\{x_1,x_2,x_3\}$.
\end{proposition}

We start by proving the following
\begin{lemma}\label{lemma:DR}
For arbitrary $\theta\in[A,B]$ we have
\begin{multline}\label{eq:lemmaDR}
\frac{f''(\theta)}{2f'(\theta)}+
\frac{f'''(\theta)}{6f'(\theta)}(x_1+x_2+x_3-3\theta)-
\left(\frac{f''(\theta)}{2f'(\theta)}\right)^2(x_2+x_3-2\theta)=\\
\frac{f''(x_1)}{2f'(x_1)}+\frac{1}{6}Sf(x_1)(x_2+x_3-2x_1)+\Obig(\Delta_\theta^{1+\beta}),
\end{multline}
where
$\Delta_\theta=\max\{x_1,x_2,x_3,\theta\}-\min\{x_1,x_2,x_3,\theta\}$.
\end{lemma}
\begin{proof} Obvious estimates
$f''(x_1)=f''(\theta)+f'''(\theta)(x_1-\theta)+\Obig(|x_1-\theta|^{1+\beta})$
and
$f'(x_1)=f'(\theta)+f''(\theta)(x_1-\theta)+\Obig((x_1-\theta)^{2})$
imply that
\begin{equation}\label{eq:lemmaDR-1}
\frac{f''(x_1)}{2f'(x_1)}=\frac{f''(\theta)}{2f'(\theta)}+\left(\frac{f'''(\theta)}{2f'(\theta)}-
\frac{(f''(\theta))^2}{2(f'(\theta))^2}\right)(x_1-\theta)+\Obig(\Delta_\theta^{1+\beta})
\end{equation}
On the other hand,
$Sf(x_1)=Sf(\theta)+\Obig(|x_1-\theta|^{\beta})$ and
$|x_2+x_3-2x_1|\le2\Delta_\theta$, hence
\begin{equation}\label{eq:lemmaDR-2}
\frac{1}{6}Sf(x_1)(x_2+x_3-2x_1)=\left(\frac{f'''(\theta)}{6f'(\theta)}-
\frac{(f''(\theta))^2}{4(f'(\theta))^2}\right)(x_2+x_3-2x_1)+\Obig(\Delta_\theta^{1+\beta})
\end{equation}
Adding (\ref{eq:lemmaDR-1}) and (\ref{eq:lemmaDR-2}) gives
(\ref{eq:lemmaDR}).
\end{proof}

\begin{remark}
Notice, that Lemma~\ref{lemma:DR}, in particular, provides an
alternative, more general (though less memorizable) formulation of
Proposition~\ref{prop:DR} as we may choose $\theta=x_2$, or $x_3$,
or any other point between $\min\{x_1,x_2,x_3\}$ and
$\max\{x_1,x_2,x_3\}$ to get the same order
$\Obig(\Delta^{1+\beta})$ as in (\ref{eq:propDR}).
\end{remark}

\begin{proof}[Proof of Proposition~\ref{prop:DR}]
Using $x_2$ as the reference point for taking derivatives, we get
$$
\frac{f(x_1)-f(x_2)}{x_1-x_2}=f'(x_2)+\frac{1}{2}f''(x_2)(x_1-x_2)+
\frac{1}{6}f'''(x_2)(x_1-x_2)^2+\Obig(|x_1-x_2|^{2+\beta}),
$$
$$
\frac{f(x_2)-f(x_3)}{x_2-x_3}=f'(x_2)+\frac{1}{2}f''(x_2)(x_3-x_2)+
\frac{1}{6}f'''(x_2)(x_3-x_2)^2+\Obig(|x_3-x_2|^{2+\beta}),
$$
and after dividing (in view of the expansion
$(1+t)^{-1}=1-t+t^2+\Obig(t^3)$) obtain
\begin{multline}\label{eq:propDR-1}
\DR(x_1,x_2,x_3;f)=1+(x_1-x_3)\left[\frac{f''(x_2)}{2f'(x_2)}+\frac{f'''(x_2)}{6f'(x_2)}(x_1+x_3-2x_2)\right.
\\
-\left.\left(\frac{f''(x_2)}{2f'(x_2)}\right)^2(x_3-x_2)\right]+\Obig(\Delta^{2+\beta})
\end{multline}

In the case when $x_2$ lies between $x_1$ and $x_3$, the estimate
(\ref{eq:propDR-1}) implies
\begin{multline}\label{eq:propDR-3}
\DR(x_1,x_2,x_3;f)=1+(x_1-x_3)\left[\frac{f''(x_2)}{2f'(x_2)}+\frac{f'''(x_2)}{6f'(x_2)}(x_1+x_3-2x_2)\right.
\\
-\left.\left(\frac{f''(x_2)}{2f'(x_2)}\right)^2(x_3-x_2)+\Obig(\Delta^{1+\beta})\right]
\end{multline}
It is not hard to notice that the expression in the square
brackets here is exactly the subject of Lemma~\ref{lemma:DR} with
$\theta=x_2$, thus (\ref{eq:propDR}) is proven.

Suppose that $x_1$ lies between $x_2$ and $x_3$. Then the version
of (\ref{eq:propDR}) for $\DR(x_2,x_1,x_3;f)$ is proven. Also, the
version of (\ref{eq:propDR-1}) for $\DR(x_1,x_3,x_2;f)$ is proven.
One can check the following exact relation takes place:
\begin{equation}\label{eq:propDR-2}
\DR(x_1,x_2,x_3;f)=1+\frac{x_1-x_3}{x_2-x_3}(\DR(x_2,x_1,x_3;f)-1)\DR(x_1,x_3,x_2;f)
\end{equation}
Substituting
$$
\DR(x_2,x_1,x_3;f)-1=(x_2-x_3)\left(\frac{f''(x_2)}{2f'(x_2)}+\frac{1}{6}Sf(x_2)(x_1+x_3-2x_2)
+\Obig(\Delta^{1+\beta})\right)
$$
and
$$
\DR(x_1,x_3,x_2;f)=1+(x_1-x_2)\frac{f''(x_2)}{2f'(x_2)}+\Obig(\Delta^{1+\beta})
$$
into (\ref{eq:propDR-2}), we get (\ref{eq:propDR-3}), and
Lemma~\ref{lemma:DR} again implies (\ref{eq:propDR}).

The case when $x_3$ lies between $x_1$ and $x_2$ is similar to the
previous one. The case when two or three among the points $x_1$,
$x_2$ and $x_3$ coincide, all considerations above are valid with
obvious alterations.
\end{proof}

\begin{proposition}\label{prop:Dist}
Let $f\in C^{3+\beta}$, $\beta\in[0,1]$, and $f'>0$ on $[A,B]$.
For any $x_1, x_2, x_3, x_4\in[A,B]$ the following estimate holds:
\begin{equation}\label{eq:propDist}
\Dist(x_1,x_2,x_3,x_4;f)=1+(x_1-x_3)\left(\frac{1}{6}(x_2-x_3)Sf(\theta)+\Obig(\Delta^{1+\beta})\right)
\end{equation}
where $\Delta=\max\{x_1,x_2,x_3,x_4\}-\min\{x_1,x_2,x_3,x_4\}$ and $\theta$ is an arbitrary point between
\linebreak
$\min\{x_1,x_2,x_3,x_4\}$ and $\max\{x_1,x_2,x_3,x_4\}$.
\end{proposition}

\begin{proof}
Proposition~\ref{prop:DR} and Lemma~\ref{lemma:DR} imply
\begin{multline*}
\DR(x_1,x_2,x_3;f)=1+(x_1-x_3)\left[\frac{f''(\theta)}{2f'(\theta)}
+\frac{f'''(\theta)}{6f'(\theta)}(x_1+x_2+x_3-3\theta)\right.\\
\left.-\left(\frac{f''(\theta)}{2f'(\theta)}\right)^2(x_2+x_3-2\theta)+\Obig(\Delta^{1+\beta})\right],
\end{multline*}
\begin{multline*}
\DR(x_1,x_4,x_3;f)=1+(x_1-x_3)\left[\frac{f''(\theta)}{2f'(\theta)}
+\frac{f'''(\theta)}{6f'(\theta)}(x_1+x_4+x_3-3\theta)\right.\\
\left.-\left(\frac{f''(\theta)}{2f'(\theta)}\right)^2(x_4+x_3-2\theta)+\Obig(\Delta^{1+\beta})\right]
\end{multline*}
Dividing the first expression by the second one accordingly to
(\ref{eq:Dist=DRtimesDR}) in view of the formula
$(1+t)^{-1}=1-t+t^2+\Obig(t^3)$, we get (\ref{eq:propDist}).
\end{proof}

\begin{remark}
Obviously enough, the estimate (\ref{eq:propDist}) can be
re-written as
\begin{equation}\label{eq:log-Dist}
\log\Dist(x_1,x_2,x_3,x_4;f)=(x_1-x_3)\left(\frac{1}{6}(x_2-x_3)Sf(\theta)+\Obig(\Delta^{1+\beta})\right)
\end{equation}
\end{remark}

\section{Circle diffeomorphisms}

\subsection{Preparations}

For an orientation-preserving homeomorphism $T$ of the unit circle $\unitcircle=\R/\Z$, its {\em rotation number}
$\rho=\rho(T)$ is the value of the limit $\lim_{i\to\infty}L_T^i(x)/i$ for a lift $L_T$ of $T$ from $\unitcircle$
onto $\R$. It is known since Poincare that rotation number is always defined (up to an additive integer) and does
not depend on the starting point $x\in\R$. Rotation number $\rho$ is irrational if and only if $T$ has no periodic
points. We restrict our attention in this paper to this case. The order of points on the circle for any trajectory
$\xi_i=T^i\xi_0$, $i\in\Z$, coincides with the order of points for the rigid rotation $R_{\rho}$. This fact is
sometimes referred to as the {\em combinatorial equivalence} between $T$ and $R_{\rho}$.

We use the {\em continued fraction} expansion for the (irrational) rotation number:
\begin{equation}
\rho=[k_1,k_2,\ldots,k_n,\ldots]=\dfrac{1}{k_1+\dfrac{1}{k_2+
\dfrac{1}{\dfrac{\cdots}{k_n+\dfrac{1}{\cdots}}}}}\in(0,1)\label{eq:cont-frac}
\end{equation}
which, as usual, is understood as a limit of the sequence of {\em rational convergents}
$p_n/q_n=[k_1,k_2,\dots,k_n]$. The positive integers $k_n$, $n\ge1$, called {\em partial quotients}, are defined
uniquely for irrational $\rho$. The mutually prime positive integers $p_n$ and $q_n$ satisfy the recurrent
relation $p_n=k_np_{n-1}+p_{n-2}$, $q_n=k_nq_{n-1}+q_{n-2}$ for $n\ge1$, where it is convenient to define
$p_0=0$, $q_0=1$ and $p_{-1}=1$, $q_{-1}=0$.

Given a circle homeomorphism $T$ with irrational $\rho$, one may consider a {\em marked trajectory} (i.e. the
trajectory of a marked point) $\xi_i=T^i\xi_0\in\unitcircle$, $i\ge0$, and pick out of it the sequence of the {\em
dynamical convergents} $\xi_{q_n}$, $n\ge0$, indexed by the denominators of the consecutive rational convergents
to $\rho$. We will also conventionally use $\xi_{q_{-\!1}}=\xi_0-1$. The well-understood arithmetical properties
of rational convergents and the combinatorial equivalence between $T$ and $R_\rho$ imply that the dynamical convergents
approach the marked point, alternating their order in the following way:
\begin{equation}
\xi_{q_{\!-1}}<\xi_{q_1}<\xi_{q_3}<\dots<\xi_{q_{2m+1}}<\dots<\xi_0<\dots<\xi_{q_{2m}}<\dots<\xi_{q_2}<\xi_{q_0}
\label{eq:dyn-conv}
\end{equation}
We define the $n$th {\em fundamental segment} $\Delta^{(n)}(\xi)$ as the circle arc $[\xi,T^{q_n}\xi]$ if $n$ is
even and $[T^{q_n}\xi,\xi]$ if $n$ is odd. If there is a marked trajectory, then we use the notations
$\Delta^{(n)}_0=\Delta^{(n)}(\xi_0)$, $\Delta^{(n)}_i=\Delta^{(n)}(\xi_i)=T^i\Delta^{(n)}_0$.

The iterates $T^{q_{n}}$ and $T^{q_{n-1}}$ restricted to $\Delta_0^{(n-1)}$ and $\Delta_0^{(n)}$ respectively are
nothing else but two continuous components of the first-return map for $T$ on the segment
$\Delta_0^{(n-1)}\cup\Delta_0^{(n)}$ (with its endpoints being identified). The consecutive images of
$\Delta_0^{(n-1)}$ and $\Delta_0^{(n)}$ until their return to $\Delta_0^{(n-1)}\cup\Delta_0^{(n)}$ cover the
whole circle without overlapping (beyond their endpoints), thus forming the $n$th {\em dynamical partition}
$$
\Pcal_n=\{\Delta_i^{(n-1)},0\le i<q_{n}\}\cup\{\Delta_i^{(n)},0\le i<q_{n-1}\}
$$
of $\unitcircle$. The endpoints of the segments from $\Pcal_n$ form the set
$$
\Xi_n=\{\xi_i,0\le i< q_{n-1}+q_{n}\}
$$

Denote by $\Delta_n$ the length of $\Delta^{(n)}(\xi)$ for the
rigid rotation $R_\rho$. Obviously enough,
$\Delta_n=|q_n\rho-p_n|$. It is well known that
$\Delta_n\sim\frac{1}{q_{n+1}}$ (here `$\sim$' means `comparable',
i.e. `$A\sim B$' means `$A=\Obig(B)$ and $B=\Obig(A)$'), thus the
Diophantine properties of $\rho\in D_\delta$ can be equivalently
expressed in the form:
\begin{equation}\label{eq:Dioph}
\Delta_{n-1}^{1+\delta}=\Obig(\Delta_n)
\end{equation}

We will also have in mind the universal exponential decay property
\begin{equation}\label{eq:root of two decay}
\frac{\Delta_{n}}{\Delta_{n-k}}\le\frac{\sqrt{2}}{(\sqrt{2})^{k}},
\end{equation}
which follows from the obvious estimates $\Delta_{n}\le\frac{1}{2}\Delta_{n-2}$ and $\Delta_{n}<\Delta_{n-1}$.

In~\cite{KT-jams} it was shown that for any diffeomorphism $T\in C^{2+\alpha}(\unitcircle)$, $T'>0$,
$\alpha\in[0,1]$, with irrational rotation number the following Denjoy-type inequality takes place:
\begin{equation}\label{eq:Denjoy:2+alpha}
(T^{q_n})'(\xi)=1+\Obig(\eps_{n,\alpha}),\quad\text{where}\quad
\eps_{n,\alpha}=l_{n-1}^\alpha+\frac{l_n}{l_{n-1}}l_{n-2}^\alpha+\frac{l_n}{l_{n-2}}l_{n-3}^\alpha+\dots+\frac{l_n}{l_0}
\end{equation}
and $l_m=\max_{\xi\in\unitcircle}|\Delta_m(\xi)|$. Notice, that this estimate does not require any Diophantine
conditions on $\rho(T)$.

Unfortunately, it is not possible to write down a corresponding stronger estimate for $T\in
C^{3+\beta}(\unitcircle)$, $\beta\in[0,1]$, without additional assumptions. We will assume that the conjugacy is
at least $C^1$-smooth: $\phi\in C^{1+\gamma}(\unitcircle)$, $\phi'>0$, with some $\gamma\in[0,1]$. (Notice, that
in conditions of Theorem~1 this assumption holds true with $\gamma=1-\delta$ accordingly to~\cite{KT-jams}, and
our aim is to raise the value of $\gamma$ to $1-\delta+\beta$.)

%Since $T\in C^3(\unitcircle)$, it follows from~\cite{KT-jams} that
%the conjugacy is $C^{2-\delta}$-smooth. Equivalently, this means
%that the invariant measure generated by $T$ has the density $h\in
%C^{1-\delta}$.

This assumption is equivalent to the following one: an invariant measure generated by $T$ has the positive density
$h=\phi'\in C^{\gamma}(\unitcircle)$. This density satisfies the homologic equation
\begin{equation}\label{eq:homol}
h(\xi)=T'(\xi)h(T\xi)
\end{equation}

The continuity of $h$ immediately implies that $h(\xi)\sim 1$, and
therefore $(T^i)'(\xi)=\frac{h(\xi)}{h(T^i\xi)}\sim 1$ and
$$
|\Delta^{(n)}(\xi)|\sim l_n\sim\Delta_n\sim\frac{1}{q_{n+1}}
$$
(due to $\Delta_n=\int_{\Delta^{(n)}(\xi)}h(\eta)\,d\eta$). By this reason, we introduce the notation
$$
E_{n,\sigma}=\sum_{k=0}^n\frac{\Delta_n}{\Delta_{n-k}}\Delta_{n-k-1}^\sigma,
$$
so that $\eps_{n,\alpha}$ in (\ref{eq:Denjoy:2+alpha}) can be replaced by $E_{n,\alpha}$ as soon as we know of the
existence of continuous $h$.

It follows also that $(T^i)'\in C^\gamma(\unitcircle)$ uniformly in $i\in\Z$, i.e.
\begin{equation}\label{eq:T^i'}
(T^i)'(\xi)-(T^i)'(\eta)=\Obig(|\xi-\eta|^\gamma),
\end{equation}
since
$(T^i)'\xi-(T^i)'\eta=\frac{h(\xi)}{h(T^i\xi)}-\frac{h(\eta)}{h(T^i\eta)}$
and $T^i\xi-T^i\eta\sim\xi-\eta$.

The additional smoothness of $T$ will be used through the following quantities:
$p_n=p_n(\xi_0)=\sum_{i=0}^{q_n-1}\frac{ST(\xi_i)}{h(\xi_i)}(\xi_i-\xi_{i+q_{n-1}})$,
$\bp_n=\bp_n(\xi_0)=\sum_{i=0}^{q_{n-1}-1}\frac{ST(\xi_{i+q_n})}{h(\xi_{i+q_n})}(\xi_{i+q_n}-\xi_{i})$. We have
\begin{equation}\label{eq:p+bp}
p_n+\bp_n=\sum_{\xi\in\Xi_n}ST(\hat\xi)\frac{\hat\xi-\xi}{h(\hat\xi)},
\end{equation}
where $\hat\xi$ denotes the point from the set $\Xi_n$ following
$\xi$ in the (circular) order
$\dots\to\xi_{q_{n-1}}\to\xi_0\to\xi_{q_n}\to\dots$. It is easy to
see that $N_n(\xi_i)=\xi_{i+q_n}$ for $0\le i<q_{n-1}$ and
$N_n(\xi_i)=\xi_{i-q_{n-1}}$ for $q_{n-1}\le i<q_n+q_{n-1}$.

In the next two subsections, we will establish certain dependencies between the Denjoy-type estimates in the
forms $(T^{q_n})'(\xi)=1+\Obig(\Delta_n^\nu)$ and $(T^{q_n})'(\xi)=1+\Obig(E_{n,\sigma})$.

\subsection{Statements that use the Hoelder exponents of $T'''$ and $h$}

In all the statements of this subsection, we assume that $T\in C^{3+\beta}$ and $h\in C^{\gamma}$,
$\beta,\gamma\in[0,1]$, but do not make any use of Diophantine properties of $\rho$.

The next lemma corresponds to the exact integral relation $\int_\unitcircle\frac{ST(\xi)}{h(\xi)}d\xi$ first
demonstrated in~\cite{SK}.
\begin{lemma}\label{lemma:T->p}
If $(T^{q_n})'(\xi)=1+\Obig(\Delta_n^\nu)$, then $p_n+\bp_n=\Obig(\Delta_{n-1}^{\min\{\beta,2\nu-1\}})$.
\end{lemma}
\begin{proof}
Using the representation
$ST=\left(\frac{T''}{T'}\right)'-\frac{1}{2}\left(\frac{T''}{T'}\right)^2$,
from (\ref{eq:p+bp}) we derive
\begin{multline*}%\label{eq:prop:T->p:2}
p_n+\bp_n=\sum_{\xi\in\Xi_n}\left[\left(\frac{T''(\hat\xi)}{T'(\hat\xi)}-\frac{T''(\xi)}{T'(\xi)}\right)
\frac{1}{h(\hat\xi)}+\Obig(|\hat\xi-\xi|^{1+\beta})\right]\\
-\frac{1}{2}\sum_{\xi\in\Xi_n}\left(\frac{T''(\xi)}{T'(\xi)}\right)^2\frac{\hat\xi-\xi}{h(\hat\xi)}
\\
=\sum_{\xi\in\Xi_n}\frac{T''(\xi)}{T'(\xi)}\left[\frac{1}{h(\xi)}-\frac{1}{h(\hat\xi)}
-\frac{1}{2}\frac{T''(\xi)}{T'(\xi)}\frac{\hat\xi-\xi}{h(\hat\xi)}\right]+\Obig(\Delta_{n-1}^\beta)
\end{multline*}
Notice that
\begin{equation}\label{eq:h(xi)-h(hatxi)}
h(\xi)-h(\hat\xi)=\Obig(|\hat\xi-\xi|^\nu)
\end{equation}
due to (\ref{eq:homol}). In particular, (\ref{eq:h(xi)-h(hatxi)}) implies that the expression in the last square
brackets is $\Obig(|\hat\xi-\xi|^\gamma)$, hence using the estimate
$T''(\xi)=\frac{T'(\hat\xi)-T'(\xi)}{\hat\xi-\xi}+\Obig(\hat\xi-\xi)$ we get
%\begin{multline*}%\label{eq:prop:T->p:2}
$$
p_n+\bp_n =\sum_{\xi\in\Xi_n}\left(\frac{T'(\hat\xi)}{T'(\xi)}-1\right)\frac{1}{\hat\xi-\xi}
\left[\frac{1}{h(\xi)}-\frac{1}{h(\hat\xi)}
-\frac{1}{2}\left(\frac{T'(\hat\xi)}{T'(\xi)}-1\right)\frac{1}{h(\hat\xi)}\right]
+\Obig(\Delta_{n-1}^{\min\{\beta,\nu\}})
$$
%+\Obig(\Delta_{n-1}^\beta)
%\end{multline*}
Now, the substitutions $T'(\xi)=\frac{h(\xi)}{h(T\xi)}$ and
$T'(\hat\xi)=\frac{h(\hat\xi)}{h(T\hat\xi)}$ transform the last
estimate (exactly) into
\begin{equation}\label{eq:prop:T->p:3}
p_n+\bp_n=\frac{1}{2}\sum_{\xi\in\Xi_n}\frac{h(\hat\xi)}{(h(\xi))^2(\hat\xi-\xi)}
\left[\left(\frac{h(\xi)}{h(\hat\xi)}-1\right)^2-\left(\frac{h(T\xi)}{h(T\hat\xi)}-1\right)^2\right]
+\Obig(\Delta_{n-1}^{\min\{\beta,\nu\}})
\end{equation}
Similarly to (\ref{eq:h(xi)-h(hatxi)}), each one of two expressions in parentheses here are
$\Obig(|\hat\xi-\xi|^\nu)$. It follows, firstly, that
\begin{multline}\label{eq:prop:T->p:4}
p_n+\bp_n=\frac{1}{2}\sum_{\xi\in\Xi_n}
\left(\frac{h(T\xi)}{h(T\hat\xi)}-1\right)^2
\left[\frac{h(T\hat\xi)}{(h(T\xi))^2(T\hat\xi-T\xi)}-\frac{h(\hat\xi)}{(h(\xi))^2(\hat\xi-\xi)}\right]\\
+\Obig(\Delta_{n-1}^{\min\{\beta,2\nu-1\}}),
\end{multline}
since, as it is easy to see, the sums in (\ref{eq:prop:T->p:3}) and in (\ref{eq:prop:T->p:4}) differ by a finite
number of terms of the order $\Obig(|\hat\xi-\xi|^{2\nu-1})$, and $2\nu-1\le\nu$. Secondly, we have
$$
\frac{h(T\hat\xi)}{(h(T\xi))^2(T\hat\xi-T\xi)}:\frac{h(\hat\xi)}{(h(\xi))^2(\hat\xi-\xi)}-1
=\frac{T'(\xi)}{T'(\hat\xi)}\cdot
\left(T'(\xi):\frac{T\hat\xi-T\xi}{\hat\xi-\xi}\right)-1=\Obig(\hat\xi-\xi),
$$
so the expressions in the square brackets in (\ref{eq:prop:T->p:4}) are bounded, and therefore the whole sum in
it is $\sum_{\xi\in\Xi_n}\Obig(|\hat\xi-\xi|^{2\nu})=\Obig(\Delta_{n-1}^{2\nu-1})$.
\end{proof}

Notice, that Lemma~\ref{lemma:T->p} does not use $\gamma$. However, the next one does.

\begin{lemma}\label{lemma:prop:T->p}
If $(T^{q_n})'(\xi)=1+\Obig(\Delta_n^\nu)$, then $p_n=\Obig(\Delta_{n-1}^{\min\{\beta,2\nu-1,\gamma\}})$.
\end{lemma}
\begin{proof}
It follows from (\ref{eq:T^i'}) that
\begin{equation}\label{eq:prop:T->p}
\frac{|\Delta_i^{(n)}|}{|\Delta_0^{(n)}|}:\frac{|\Delta_i^{(n-2)}|}{|\Delta_0^{(n-2)}|}=1+\Obig(\Delta_{n-2}^\gamma)
\end{equation}
This implies, together with (\ref{eq:h(xi)-h(hatxi)}) and $ST(\xi_{i+q_n})-ST(\xi_i)=\Obig(\Delta_n^{\beta})$,
that
$$
\bp_n+\frac{|\Delta_0^{(n)}|}{|\Delta_0^{(n-2)}|}p_{n-1}
=\sum_{i=0}^{q_{n-1}-1}\Obig(\Delta_n(\Delta_{n-2}^\gamma+\Delta_n^\beta+\Delta_n^\nu))
=\frac{\Delta_n}{\Delta_{n-2}}\Obig(\Delta_{n-2}^{\min\{\beta,\gamma,\nu\}})=\Obig(\Delta_{n}^{\min\{\beta,\gamma,\nu\}})
$$
In view of this, Lemma~\ref{lemma:T->p} implies
$p_n=\frac{|\Delta_0^{(n)}|}{|\Delta_0^{(n-2)}|}p_{n-1}+\Obig(\Delta_{n-1}^{\mu})$, where
$\mu=\min\{\beta,2\nu-1,\gamma\}\le1$. Telescoping the last estimate, we get
$$
p_n=\sum_{k=0}^n\frac{|\Delta_0^{(n)}|\cdot|\Delta_0^{(n-1)}|}{|\Delta_0^{(n-k)}|\cdot|\Delta_0^{(n-k-1)}|}
\Obig(\Delta_{n-k-1}^\mu) =\Obig\left(\Delta_{n-1}^\mu
\sum_{k=0}^n\frac{\Delta_n}{\Delta_{n-k}}\left(\frac{\Delta_{n-1}}{\Delta_{n-k-1}}\right)^{1-\mu}\right),
$$
and the latter sum is bounded due to (\ref{eq:root of two decay}).
\end{proof}

\begin{lemma}\label{lemma:Dist-T^qn}
If $p_n=\Obig(\Delta_{n-1}^\omega)$, where $\omega\in[0,1]$, then
\begin{gather*}
\Dist(\xi_0,\xi,\xi_{q_{n-1}},\eta;T^{q_n})=1+(\xi-\eta)\Obig(\Delta_{n-1}^{\min\{\beta,\gamma,\omega\}}),
\quad \xi,\eta\in\Delta^{(n-1)}_0;\\
\Dist(\xi_0,\xi,\xi_{q_n},\eta;T^{q_{n-1}})=1+(\xi-\eta)\frac{\Delta_n}{\Delta_{n-2}}
\Obig(\Delta_{n-2}^{\min\{\beta,\gamma,\omega\}}), \quad \xi,\eta\in\Delta^{(n-2)}_0
\end{gather*}
\end{lemma}
\begin{proof}
Accordingly to (\ref{eq:log-Dist}) and (\ref{eq:mult-Dist}), we
have
$$
\log\Dist(\xi_0,\xi,\xi_{q_{n-1}},\eta;T^{q_n})=\frac{1}{6}\sum_{i=0}^{q_n-1}(\xi_i-\xi_{i+q_{n-1}})(T^i\xi-T^i\eta)
ST(\xi_i)+(\xi-\eta)\Obig(\Delta_{n-1}^\beta)
$$
On the other hand,
\begin{multline*}
\sum_{i=0}^{q_n-1}(\xi_i-\xi_{i+q_{n-1}})(T^i\xi-T^i\eta)ST(\xi_i)-h(\xi_0)(\xi-\eta)p_n\\
=(\xi-\eta)\sum_{i=0}^{q_n-1}(\xi_i-\xi_{i+q_{n-1}})ST(\xi_i)\left[\frac{T^i\xi-T^i\eta}{\xi-\eta}
-(T^i)'(\xi_0)\right] =(\xi-\eta)\Obig(\Delta_{n-1}^\gamma)
\end{multline*}
because of (\ref{eq:T^i'}). The first estimate of the lemma
follows. To prove the second one, we similarly notice that
$$
\log\Dist(\xi_0,\xi,\xi_{q_{n}},\eta;T^{q_{n-1}})=\frac{1}{6}\sum_{i=0}^{q_{n-1}-1}(\xi_i-\xi_{i+q_{n}})(T^i\xi-T^i\eta)
ST(\xi_i)+(\xi-\eta)\Obig(\Delta_{n-1}^\beta)
$$
and
\begin{multline*}
\sum_{i=0}^{q_{n-1}-1}(\xi_i-\xi_{i+q_{n}})(T^i\xi-T^i\eta)ST(\xi_i)
-h(\xi_0)(\xi-\eta)\frac{|\Delta_0^{(n)}|}{|\Delta_0^{(n-2)}|}p_{n-1}\\
=(\xi-\eta)\sum_{i=0}^{q_{n-1}-1}(\xi_i-\xi_{i+q_{n}})ST(\xi_i)\left[\frac{T^i\xi-T^i\eta}{\xi-\eta}
-(T^i)'(\xi_0)\frac{|\Delta_i^{(n-2)}|}{|\Delta_0^{(n-2)}|}:\frac{|\Delta_i^{(n)}|}{|\Delta_0^{(n)}|}\right]\\
=(\xi-\eta)\sum_{i=0}^{q_{n-1}-1}(\xi_i-\xi_{i+q_{n}})ST(\xi_i)\Obig(\Delta_{n-2}^\gamma)
=(\xi-\eta)\frac{\Delta_n}{\Delta_{n-2}}\Obig(\Delta_{n-2}^\gamma)
\end{multline*}
(see (\ref{eq:prop:T->p})).
\end{proof}

%\begin{proof}
%\end{proof}

As in~\cite{KT-jams}, we introduce the functions
$$
M_n(\xi)=\DR(\xi_0,\xi,\xi_{q_{n-1}};T^{q_n}),\quad \xi\in\Delta_0^{(n-1)};
$$
$$
K_n(\xi)=\DR(\xi_0,\xi,\xi_{q_{n}};T^{q_{n-1}}),\quad \xi\in\Delta_0^{(n-2)},
$$
where $\xi_0$ is arbitrarily fixed. The following three exact
relations can be easily checked:

\begin{equation}\label{eq:observ1}
M_n(\xi_0)\cdot M_n(\xi_{q_{n-1}})=K_n(\xi_0)\cdot K_n(\xi_{q_n}),
\end{equation}
%where $m_n=\exp\{(-1)^n\sum_{i=0}^{q_n-1}\int_{\Delta_i^{(n-1)}}\frac{T''(\xi)}{2T'(\xi)}d\xi\}$;

\begin{equation}\label{eq:observ2}
K_{n+1}(\xi_{q_{n-1}})-1=\frac{|\Delta_0^{(n+1)}|}{|\Delta_0^{(n-1)}|}\left(M_n(\xi_{q_{n+1}})-1\right),
\end{equation}

\begin{equation}\label{eq:observ3}
\frac{(T^{q_{n+1}})'(\xi_0)}{M_{n+1}(\xi_0)}-1=\frac{|\Delta_0^{(n+1)}|}{|\Delta_0^{(n)}|}
\left(1-\frac{(T^{q_{n}})'(\xi_0)}{K_{n+1}(\xi_0)}\right)
\end{equation}

Also notice that
\begin{equation}\label{eq:M/M,K/K}
\frac{M_n(\xi)}{M_n(\eta)}=\Dist(\xi_0,\xi,\xi_{q_{n-1}},\eta;T^{q_n}),\qquad
\frac{K_n(\xi)}{K_n(\eta)}=\Dist(\xi_0,\xi,\xi_{q_{n}},\eta;T^{q_{n-1}})
\end{equation}

\begin{lemma}\label{lemma:prop:p->T}
If $p_n=\Obig(\Delta_{n-1}^\omega)$, $\omega\in[0,1]$, then
$(T^{q_n})'(\xi)=1+\Obig({E}_{n,1+{\min\{\beta,\gamma,\omega\}}})$.
%, where
%$\eps_{n,1+\omega}=\Delta_{n-1}^{1+\omega}+\frac{\Delta_n}{\Delta_{n-1}}\Delta_{n-2}^{1+\omega}+
%\frac{\Delta_n}{\Delta_{n-2}}\Delta_{n-3}^{1+\omega}+\dots+\frac{\Delta_n}{\Delta_0}$.
\end{lemma}

\begin{proof}
Let $\sigma=1+{\min\{\beta,\gamma,\omega\}}$. In view of (\ref{eq:M/M,K/K}), Lemma~\ref{lemma:Dist-T^qn} implies
that $M_n(\xi)/M_n(\eta)=1+\Obig(\Delta_{n-1}^{\sigma+1})$ and
$K_n(\xi)/K_n(\eta)=1+\Obig(\Delta_n\Delta_{n-2}^{\sigma})$. In our assumptions, the functions $M_n(\xi)\sim1$ and
$K_n(\xi)\sim1$, since $(T^i)'(\xi)\sim1$. This gives us
\begin{equation}\label{eq:observ1res}
M_n(\xi)=m_n+\Obig(\Delta_{n-1}^{\sigma+1}),\qquad K_n(\xi)=m_n+\Obig(\Delta_n\Delta_{n-2}^{\sigma})
\end{equation}
where $m_n^2$ denotes the products in (\ref{eq:observ1}). Due to (\ref{eq:observ2}) and (\ref{eq:observ1res}) we
have
\begin{equation}\label{eq:observ1sub}
m_{n+1}-1=\frac{|\Delta_0^{(n+1)}|}{|\Delta_0^{(n-1)}|}(m_n-1)+\Obig(\Delta_{n+1}\Delta_{n-1}^{\sigma}),
\end{equation}
which is telescoped into $m_n-1=\Obig(\Delta_n{E}_{n-1,\sigma-1})$, which in turn implies
\begin{equation}\label{eq:observ2res}
M_n(\xi)=1+\Obig(\Delta_{n-1}{E}_{n,\sigma-1}),\qquad K_n(\xi)=1+\Obig(\Delta_{n}{E}_{n-1,\sigma-1})
\end{equation}
(notice that $\Delta_{n-1}{E}_{n,\sigma-1}=\Delta_{n-1}^{1+\sigma}+\Delta_{n}{E}_{n-1,\sigma-1}$). Due to
(\ref{eq:observ2}) and (\ref{eq:observ2res}) we have
\begin{equation}\label{eq:observ3sub}
(T^{q_{n+1}})'(\xi_0)-1=\frac{|\Delta_0^{(n+1)}|}{|\Delta_0^{(n)}|}(1-(T^{q_{n}})'(\xi_0))
+\Obig(\Delta_{n}{E}_{n+1,\sigma-1})
\end{equation}
which is telescoped into
\begin{multline*}
(T^{q_{n}})'(\xi_0)-1=\Obig\left(\sum_{k=0}^n\frac{\Delta_n}{\Delta_{n-k}}\Delta_{n-k-1}{E}_{n-k,\sigma-1}\right)\\
=\Obig\left(\Delta_n\sum_{k=0}^n\sum_{m=0}^{n-k}\frac{\Delta_{n-k-1}}{\Delta_{n-k-m}}\Delta_{n-k-m-1}^\sigma\right)
=\Obig\left(\Delta_n\sum_{k=0}^n\sum_{s=k}^{n}\frac{\Delta_{n-k-1}}{\Delta_{n-s}}\Delta_{n-s-1}^\sigma\right)\\
=\Obig\left(\Delta_n\sum_{s=0}^n\frac{\Delta_{n-s-1}^\sigma}{\Delta_{n-s}}\sum_{k=0}^{s}\Delta_{n-k-1}\right)
=\Obig({E}_{n,\sigma}),
%=\Obig\left(\Delta_n\sum_{s=0}^n\frac{\Delta_{n-s-1}^\sigma}{\Delta_{n-s}}\Delta_{n-s-1}\right)=\Obig({E}_{n,1+\sigma}),
\end{multline*}
since $\sum_{k=0}^{s}\Delta_{n-k-1}=\Obig(\Delta_{n-s-1})$ due to (\ref{eq:root of two decay}).
\end{proof}

The summary of this subsection is given by
\begin{proposition}\label{prop:T->T}
Suppose that for a diffeomorphism $T\in C^{3+\beta}(\unitcircle)$, $T'>0$, $\beta\in[0,1]$, with irrational
rotation number there exists density $h\in C^\gamma(\unitcircle)$, $\gamma\in[0,1]$, of the invariant measure and
the following asymptotical estimate holds true: $(T^{q_n})'(\xi)=1+\Obig(\Delta_n^\nu)$ with certain real constant
$\nu$. Then $(T^{q_n})'(\xi)=1+\Obig(E_{n,1+\min\{\beta,\gamma,2\nu-1\}})$.
\end{proposition}
\begin{proof}
Follows from Lemmas~\ref{lemma:prop:T->p} and~\ref{lemma:prop:p->T} immediately.
\end{proof}

\begin{remark}
In~\cite{Y} it is shown that for any $T\in C^3(\unitcircle)$ the following Denjoy-type estimate takes place:
$(T^{q_n})'(\xi)=1+\Obig(l_n^{1/2})$, and in our assumptions it is equivalent to
$(T^{q_n})'(\xi)=1+\Obig(\Delta_n^{1/2})$. Hence, in fact we have $\nu\ge\frac{1}{2}$, though this is of no use
for us.
\end{remark}

\subsection{Statements that use Diophantine properties of $\rho$}

Now we start using the assumption $\rho\in D_\delta$, $\delta\ge0$, however forget about the smoothness of $T$ and
the Hoelder condition on $h$.

\begin{lemma}\label{lemma:T->h}
If $(T^{q_n})'(\xi)=1+\Obig(\Delta_n^{\nu})$, $\nu\in\left[\frac{\delta}{1+\delta},1\right]$, then $h\in
C^{\nu(1+\delta)-\delta}(\unitcircle)$.
\end{lemma}
\begin{proof}
Consider two points $\xi_0,\xi\in\unitcircle$ and $n\ge0$ such that
$\Delta_n\le|\phi(\xi)-\phi(\xi_0)|<\Delta_{n-1}$. Let $k$ be the greatest positive integer such that
$|\phi(\xi)-\phi(\xi_0)|\ge k\Delta_n$. (It follows that $1\le k\le k_{n+1}$.) Due to the combinatorics of
trajectories, continuity of $h$ and the homologic equation (\ref{eq:homol}), we have
$$
\log h(\xi)-\log h(\xi_0)=\Obig\left(k\Delta_n^{\nu} +\sum_{s=n+1}^{+\infty}k_{s+1}\Delta_s^{\nu}\right),
$$
and the same estimate holds for $h(\xi)-h(\xi_0)$, since $\log h(\xi)=\Obig(1)$.

We have $k_{n+1}<\Delta_{n-1}/\Delta_n=\Obig\bigl(\Delta_n^{-\frac{\delta}{1+\delta}}\bigr)$, hence
$$
k\Delta_n^\nu=k^{\nu(1+\delta)-\delta}\Delta_n^{\nu(1+\delta)-\delta} \cdot
k^{(1+\delta)(1-\nu)}\Delta_n^{\delta(1-\nu)}=\Obig\left((k\Delta_n)^{\nu(1+\delta)-\delta}\right)
$$
and
$$
\sum_{m=n+1}^{+\infty}k_{m+1}\eps_m
=\Obig\left(\sum_{m=n+1}^{+\infty}\Delta_{m}^{\frac{\nu(1+\delta)-\delta}{1+\delta}}\right)
=\Obig\left(\sum_{m=n+1}^{+\infty}\Delta_{m-1}^{\nu(1+\delta)-\delta}\right)
=\Obig\left(\Delta_{n}^{\nu(1+\delta)-\delta}\right)
$$
due to (\ref{eq:Dioph}) and (\ref{eq:root of two decay}). Finally, we obtain
$$
h(\xi)-h(\xi_0)=\Obig((k\Delta_n)^{\nu(1+\delta)-\delta})=\Obig(|\phi(\xi)-\phi(\xi_0)|^{\nu(1+\delta)-\delta})=
\Obig(|\xi-\xi_0|^{\nu(1+\delta)-\delta})$$
\end{proof}

\begin{lemma}\label{lemma:E}
If $\sigma\in[0,1+\delta)$, then $E_{n,\sigma}=\Obig\bigl(\Delta_n^{\frac{\sigma}{1+\delta}}\bigr)$.
\end{lemma}
\begin{proof}
Due to (\ref{eq:Dioph}) we have
$$
E_{n,\sigma}=\Obig\left(\Delta_n\sum_{k=0}^n\Delta_{n-k}^{\frac{\sigma}{1+\delta}-1}\right)%=\Obig(\Delta_n^{\frac{s}{1+\delta}}),
$$
The statement of the lemma follows, since
$\sum_{k=0}^n\Delta_{n-k}^{\frac{\sigma}{1+\delta}-1}=\Obig\bigl(\Delta_n^{\frac{\sigma}{1+\delta}-1}\bigr)$
because of (\ref{eq:root of two decay}).
\end{proof}

This subsection is summarized by

\begin{proposition}\label{prop:T->h}
Suppose that for a diffeomorphism $T\in C^1(\unitcircle)$, $T'>0$, with rotation number $\rho\in D_\delta$,
$\delta\ge0$, there exists a continuous density $h$ of the invariant measure, and the following asymptotical
estimate holds true: $(T^{q_n})'(\xi)=1+\Obig(E_{n,\sigma})$ with certain constant $\sigma\in[0,1+\delta)$. Then
$(T^{q_n})'(\xi)=1+\Obig\bigl(\Delta_n^\frac{\sigma}{1+\delta}\bigr)$ and $h\in
C^{\max\{0,\sigma-\delta\}}(\unitcircle)$.
\end{proposition}
\begin{proof}
Follows from Lemmas~\ref{lemma:E} and~\ref{lemma:T->h} immediately.
\end{proof}

\subsection{Proof of Theorem~1}

Recall that we need to prove Theorem~1 for $r=3+\beta$, $0<\beta<\delta<1$. We will use a finite inductive
procedure based on Propositions~\ref{prop:T->T} and~\ref{prop:T->h} to improve step by step the Denjoy-type
estimate in the form
\begin{equation}\label{eq:theorem:Denjoy}
(T^{q_n})'(\xi)=1+\Obig(E_{n,\sigma})
\end{equation}
From~\cite{KT-jams}, it follows that (\ref{eq:theorem:Denjoy}) holds true for $\sigma=1$ (see
(\ref{eq:Denjoy:2+alpha})), so this will be our starting point.  Consider the sequence $\sigma_0=1$,
$\sigma_{i+1}=\min\left\{1+\beta,\frac{2}{1+\delta}\sigma_i\right\}$, $i\ge0$. The inductive step is given by the
following
\begin{lemma}\label{lemma:theorem}
Suppose that $\sigma_i\in[1,1+\beta]$ and (\ref{eq:theorem:Denjoy}) holds for $\sigma=\sigma_i$. Then
$\sigma_{i+1}\in[1,1+\beta]$ and (\ref{eq:theorem:Denjoy}) holds for $\sigma=\sigma_{i+1}$.
\end{lemma}
\begin{proof}
First of all, notice that $\sigma_i<1+\delta$ since $\beta<\delta$. Proposition~\ref{prop:T->h} implies that $h\in
C^{\gamma_i}(\unitcircle)$ with $\gamma_i=\sigma_i-\delta\in(0,1)$ and
$(T^{q_n})'(\xi)=1+\Obig(\Delta_n^{\nu_i})$ with $\nu_i=\frac{\sigma_i}{1+\delta}\in(0,1)$.
Proposition~\ref{prop:T->T} then implies that (\ref{eq:theorem:Denjoy}) holds for
$\sigma=\min\{1+\beta,1+\gamma_i,2\nu_i\}$, and this is exactly $\sigma_{i+1}$ since
$1+\sigma_i-\delta>\frac{2\sigma_i}{1+\delta}$ (indeed,
$(1+\sigma_i-\delta)(1+\delta)-2\sigma_i=(1-\delta)(1+\delta-\sigma_i)>0$). The bounds on $\sigma_{i+1}$ are easy
to derive.
\end{proof}
What is left is to notice that $\sigma_i=\min\left\{1+\beta,\left(\frac{2}{1+\delta}\right)^i\right\}$, $i\ge0$,
where $\frac{2}{1+\delta}>1$, so this sequence reaches $1+\beta$ in a finite number of steps. And as soon as
(\ref{eq:theorem:Denjoy}) with $\sigma=1+\beta$ is shown, Proposition~\ref{prop:T->h} implies that $h\in
C^{1+\beta-\delta}$. Theorem~1 is proven.

{\em Acknowledgement.} The author thanks Konstantin Khanin for inspiration and lots of useful discussions.


\begin{center}
\bf References
\end{center}


\begin{enumerate}

\bibitem{KT-jams} K.~Khanin and A.~Teplinsky. Herman's theory
revisited, {\tt arXiv:math.DS/0707.0075}.

\bibitem{H} M.-R.~Herman. Sur la conjugaison differentiable des diffeomorphismes du cercle a des rotations,
{\em I.\ H.\ E.\ S.\ Publ.\ Math.}, {\bf 49}, 5--233 (1979).

\bibitem{Y} J.-C.~Yoccoz. Conjugaison differentiable des diffeomorphismes du cercle dont le nombre de rotation
verifie une condition diophantienne. {\em Ann.\ Sci.\ Ecole Norm.\ Sup.\ (4)}, {\bf 17}:3, 333--359 (1984).

\bibitem{KO1} Y.~Katznelson and D.~Ornstein. The differentiability of the conjugation of certain diffeomorphisms
of the circle, {\em Ergodic Theory Dynam. Systems}, {\bf 9}:4, 643--680 (1989).

\bibitem{SK} Ya.~G.~Sinai and K.~M.~Khanin, Smoothness of conjugacies of diffeomorphisms of the circle with rotations,
{\em Uspekhi Mat.~Nauk} {\bf 44}:1, 57–-82 (1989), in Russian; English transl., {\em Russian Math.~Surveys} {\bf
44}:1, 69–-99 (1989).

\end{enumerate}
\end{document}